\documentclass[12pt]{amsart}

\usepackage{color}

\usepackage{dsfont}
\usepackage{amsmath}
\usepackage{amssymb}
\usepackage{graphicx}

\usepackage{amssymb}

\usepackage{amsfonts}

\usepackage{soul}

\usepackage{mathpazo}
\usepackage{color}
\usepackage{yfonts}

\usepackage{paralist}

\usepackage{stmaryrd}

\usepackage{amsxtra}

\newtheorem*{theorem*}{Theorem}
\def\carre{\hfill $\Box$}

\def\a{          \alpha}

\def \R{{\mathbb R}}

\def \Z{{\mathbb Z}}
\def \N{{\mathbb N}}

\newcommand{\T}{{\mathbb T}}
\newcommand{\prf}{{\begin{proof}}}
\newcommand{\epf}{{\end{proof}}}

\newcommand{\Q}{{\mathbb Q}}

\newtheorem{theo}{\sc Theorem}
\newtheorem{prop}{\sc Proposition}

\theoremstyle{definition}

\def\bee{\begin{equation}}
\def\eee{\end{equation}}

\newtheorem{defi}{\sc Definition}

\theoremstyle{remark}

\newcommand{\pdvr}[2]
{\dfrac{\partial^{#2} #1}{\partial \theta^{#2_1} \partial r^{#2_2}}}
\newcommand{\pdvrs}[2]
{\partial^{#2} #1 /\partial \theta^{#2_1} \partial r^{#2_2}}
\theoremstyle{definition}

\newtheorem{rem}{Remark}%[section]
\numberwithin{equation}{section}

\definecolor{orange}{rgb}{1,0.5,0}

\title[Continuous spectrum for mixing Schr\"odinger operators]{Continuous spectrum for a class of smooth mixing Schr\"odinger operators}

\author[Bassam Fayad and Yanhui Qu]{Bassam Fayad$^1$ and Yanhui Qu$^2$}
\thanks{ $^1$  Supported by  ANR-15-CE40-0001 and by the project BRNUH
}
\thanks{$^2$ Supported by National Natural Science Foundation of China,  No. 11371055 and No. 11431007.}

\address[Y.H. Qu]{Department of  Mathematical Science, Tsinghua University, Beijing 100084, P. R. China}
\email{yhqu@math.tsinghua.edu.cn}

\address[B. Fayad]{CNRS, IMJ-PRG, France}
\email{bassam.fayad@gmail.com}
\begin{document}

\begin{abstract} We give the first example of a smooth volume preserving mixing dynamical system such that the discrete Schr\"odin\-ger operators on the line defined with a potential generated by this system and a H\"older sampling function, have almost surely a continuous spectrum. 
\end{abstract}

\maketitle

\section{Introduction}

Given a dynamical system $(\Omega,T,\mu)$, a sample function  $V : \Omega  \to \R$ and a base point $x\in \Omega,$ we define the 1d Schr\"odinger operator generated by $(\Omega,T,\mu)$,  $V$ and $x$  as the operator on $\ell^2(\Z)$
\begin{equation} \label{operator} \tag{$*$} (H_{T,V,x}u)_n= u_{n+1}+u_{n-1} +V(T^n x) u_n. \end{equation}

A general fact in spectral theory of 1d Schr\"odinger operators is that randomness of the potential is a source of localization of the spectrum. Thus an ergodic dynamical system with randomness features has in general a localized pure point spectrum.

 The goal of this paper is to prove the following result that we state informally here before we give its exact statement at the end of this seciton. 

\medskip 

\noindent {\bf Theorem. } {\it There exist smooth volume preserving and mixing dynamical systems such that the associated 1d Schr\"odinger operators with H\"older potentials have, for almost every base point, no pure point in their spectrum. }

\medskip

Let us first recall some of the instances where randomness yields pure point spectra.  The most famous example is Anderson's model: the dynamical system is the Bernoulli shift $(\R^\Z,\sigma,\mu^\Z)$, where   $\mu$ is a probability measure  supported on $\R$; the sample function $V: \R^\Z\to\R$ is defined by $V(x)=x_0$, where $x=\cdots x_{-1}x_0x_1\cdots.$ It is well-known that if $\mu=r(x)dx$ with $r$ bounded and compactly supported,
%under suitable assumption {be a bit more clear about the "suitable assumptions" is it something like $V(\cdot)$  has a bounded density  and compact support?} on $X$ and $\mu$, 
 then for $\mu^\Z$ a.e. $x\in \R^\Z$, 
the operator $H_{\sigma,V,x}$ has pure point spectrum and the related  eigenfunctions are localized, see for example \cite{KS,CFKS,CL,PF}.  Another kind of example is given by Bourgain and Schlag \cite{BS}.  In their example, the dynamical system is $(\T^2,A,{\rm Leb})$, where $A:\T^2\to \T^2$ is a hyperbolic toral automorphism; the sample function $V(x)=\lambda F(x)$ with $F\in C^1(\T^2)$ non-constant and $\int F=0$. For  $\delta>0$ small and  the coupling constant $\lambda>0$ sufficiently small, they established Anderson localization on some subinterval $I_0$ of 
$[-2+\delta,-\delta]\cup[\delta,2-\delta]$.
%the spectrum up to edges and the center {be a bit more clear about "up to edges and center"}. 
We note that $(A,\T^2,{\rm Leb})$ is mixing since $A$ is hyperbolic. 

An interesting question is to understand how much randomness is needed to insure localization. 
For example,  recently much interest in
Schr\"odinger operators generated by the skew-shift was in part motivated by the study of the transition from complete localization to continuous spectra as randomness of the potential decreases.
 
  It is a known fact indeed that quasi-periodic potentials often display absolutely continuous spectra for small coupling. The skew-shift dynamics are on one hand related to the quasi-periodic dynamics, but present on the other hand some mixing (in the fibers) and parabolic features (see \cite{katok}).
The skew-shift $T_{\alpha}$ is defined on $\T^2$ by $T_\alpha(x,y)=(x+\alpha, y+x)$, where $\alpha\in \R$ is irrational. When $\a$ is Diophantine and the sample function is regular, it is expected that localization holds almost surely for arbitrary nonzero coupling constant (see \cite{BGS,B} for results in that direction).

%{\color{red}\footnote{improve, extend and give references...} ???

%Consider  For Diophantine $\alpha,$Bourgain,
%Goldstein, and Schlag \cite{BGS} proved a localization result for analytic and sufficiently
%large sample function  and other localization results were subsequently obtained. 

We note that the system $( \T^2,T_\alpha,{\rm Leb})$ is strictly ergodic, but not weakly mixing. Thus, the above expectation aims at establishing localization under weak randomness hypothesis. 

Our approach in the present work is a dual one, since we will construct a smooth mixing system for which the 
associated Schr\"odinger operators with any H\"older potential have no eigenvalues for almost all base points. 
 
 Our construction is a smooth reparametrization of a linear flow on $\T^3$ (see Section \ref{sec.rep} for the definitions), that will combine mixing 
with the existence of super-recurrence times (see Definition \ref{superrecurrent}) for almost every point.  The strong recurrence implies a Gordon property on the potential, that in turn yields absence of pure point in the spectrum (see Section \ref{sec.gordon} below).

The construction of the reparametrized flow follows the construction in \cite{singular} of a reparametrization of a linear flow on $\T^3$ that is mixing but has a purely singular maximal spectral type. Indeed, the singularity of the maximal spectral type in \cite{singular} was due to the existence of very strong periodic approximations on  parts of the phase space that have a slowly decaying measure.

Note that for continuous potentials, the behavior of the Schr\"odinger operator above the skew shifts is completely different from the one described in the above conjecture. Indeed, Boshernitzan and Damanik showed that for a typical skew-shift and a  generic continuous sampling function, the associated Schr\"odinger
operators have no eigenvalues for almost all base points \cite{BD1}. In \cite{BD2}, they generalized this result to skew shifts on the torus $\T^3$ with a frequency in a residual set. They also showed that their approach cannot be carried out for skew shifts in dimension larger or equal to $4$. 

The approach of \cite{BD1,BD2} is to introduce a recurrence property called "Repetition Property",  %(MRP) (see Definition \ref{repetition}) 
 that, when  checked for a given dynamical system (at some or all orbits depending on the system)  implies a Gordon property on the generic continuous potential  (above some or all orbits depending on the system).
   Our approach is inspired by theirs, but in our case the Gordon property follows directly from super-recurrence, for all H\"older potentials. In their case, especially since they include almost every frequency of the skew shift, recurrence is not very strong and the "Repetition Property" they prove is only sufficient to guarantee continuous spectrum for a generic continuous potential.
Note also that in \cite{BD1}, the "Repetion Property" and its spectral consequences are shown to hold  for almost every point, for almost every interval exchange.  It is well-known that almost every interval exchange transformation is weakly mixing \cite{AF}, but it is never mixing \cite{K}.

\medskip

We will now give the precise statement of our construction. We start with the definition of super-recurrence.

Let $(\Omega,T)$ be a topological dynamical system with $\Omega$
a compact metric space and $T$ a homeomorphism.

\begin{defi} \label{superrecurrent} 
Assume $x\in \Omega$. If there exist $\alpha>1$ and an  integer
sequence $k_n\uparrow \infty$ such that
$$
d(T^{k_n}x,x)\le \exp(-k_n^{\alpha}),
$$
then we say that $x$ is super-recurrent with recurrent exponent
$\alpha$.

If $\mu$ is an invariant ergodic measure of $T$, we say that the system $(\Omega,T,\mu)$  is super-recurrent if $\mu$-a.e. $x \in \Omega$ is super-recurrent. 
\end{defi}

%A natural rising question is that what is the borderline  of randomness  for the dynamical system  to produce continuous spectrum. Still in \cite{BD1}, they 
%showed that almost every interval exchange transformation has (MRP) relative to Lebesgue measure. As a consequence, for generic continuous sampling function, the associated Schr\"odinger
%operators have no eigenvalues for almost all base points.In \cite{HXY}, Huang et. al. constructed a positive entropy minimal dynamical system with (TRP) (see Definition \ref{repetition}).  As a consequence, for generic continuous sampling function, the associated Schr\"odinger
%operators have no eigenvalues for  points in a residual subset.

%In a more general setting, Avila and Damanik \cite{AD} showed that if $f$ is a homeomorphism of a compact space $X$ to itself then for generic  continuos potential $V:X\to \R$, the spectrum of the  related operator has no  absolutely continuous component.  

%We are interested here in the same problem of the relation between randomness of the driving dynamical system and continuous  spectrum of the associated Schr\"odinger operators.  
%Using reparametrizations of linear flows, we construct examples on the three torus of mixing uniquely ergodic volume preserving flows for which the associated  Schr\"odinger operators with a H\"older sample function have no point part in the spectrum for almost all base points.

 As mentioned above,  our examples will be  reparametrization of a minimal translation flow on the three torus by a smooth function $\Phi$.  As will be recalled in Section \ref{three}, such flows are uniquely ergodic for a measure equivalent to the Haar measure with
density ${1 \over \phi}$. We denote by $\mu$ the Haar measure on the torus and by  $\mu_\phi$ the measure with density ${1 \over \phi}$. Note that reparametrizations of linear flows always have zero topological entropy.

\begin{theo} 
\label{application} There exists $(\a,\a') \in \R^2$ and a smooth  reparametrization $\phi \in C^\infty(\T^3,{\R_+^*})$ of the translation flow $T_{t(\a,\a',1)}$ such that the resulting flow is mixing, for its unique ergodic invariant probability measure $\mu_\phi$,  and $\mu$ a.e. $x\in \T^3$ is super-recurrent for its time one map $T$. 

As a consequence, 
%$(\T^3,T, \mu)$  satisfies (MRP). Moreover 
for every H\"older continuous potential $V:\T^3 \to \R$, the operator $H_{T,V,x}$ has purely continuous spectrum for $\mu$ a.e. $x\in \T^3$. 
\end{theo}

\begin{rem}
{\rm  It is easy to see that super-recurrence for almost every point implies the MRP property of \cite{BD1}. Hence for generic continuous function $V$ and $\mu$-a.e. $x$, the operator $H_{T,V,x}$ has continuous spectrum.
%(i) See Definition \ref{superrecurrent} for the definition of super-recurrence on a topological dynamical system $(\Omega,T,d)$.   When $T$ has H\"older property, it  induces strongly recurrent property like (MRP).

%(ii) By \cite{BD1}, for generic continuous function $V$ and $\mu$-a.e. $x$, the operator $H_{T,V,x}$ has continuous spectrum, however for a concrete $V$, it is hard to see whether it is in that residual set. On the other hand, if $V$ is more regular, say H\"older continuous, we can directly show that $\{V(T^nx)\}$ is a Gordon potential for $\mu$-a.e. $x$,  as a consequence, the operator $H_{T,V,x}$ has continuous spectrum. 

%(iii) We note that the system  $(\T^3,T,\mu)$ is strictly ergodic, mixing. It has  zero topological entropy, nevertheless, it is topologically mixing.
}
\end{rem}

%To prove the super-recurrence we follow the same strategy as \cite{BD1,BD2} based on recurrence. Indeed, o
The frequencies $\a$ and $\a'$ are  specially chosen super Liouville numbers, so that $T_{t_n(\a,\a',1)}$ is very close to identity for some sequence $t_n \to \infty$ (see Section \ref{R}). Naturally, the very strong periodic approximations of the linear flow are lost after time change, otherwise mixing would not be possible. However, one can choose the reparametrization in such a way that along a sequence of times $t_n \to \infty$, the very strong almost periodic behavior  of the translation flow still appears on a set of small measure $\epsilon_n$. If now $\epsilon_n$ decreases, but not too rapidly, say $\epsilon_n \sim \frac{1}{n}$, then by a Borel-Cantelli argument, most of the points on the torus  will be strongly recurrent along a subsequence of the sequence $t_n$.
% This results in a Gordon property of any H\"older potential, a property that overrules the existence of eigenvalues for the associated Schr\"odinger operator.  

\section{Super-recurrence and continuous spectrum}
\label{sec.gordon}

To show why super-recurrence implies the absence of a point part in the spectrum, we just have to show that it implies the Gordon condition, that we now state. 
%By take a further subsequence if necessary, we may and will assume $k_n\ge n$ from now on.

\noindent {\bf The Gordon condition.} A bounded function $V:\Z\to\R$ is called a Gordon potential if there
are positive integers   $k_n\to \infty$ such that
$$
  \max_{1\le l\le k_n}|V(l)-V(l\pm k_n)|\le n^{-k_n}
$$
for any $n\ge1.$  

The Gordon condition insures that the  1d Schr\"odinger operator on $\ell^2(\Z)$ with potential $V$ has no eigenvalues \cite{gordon}.

Assume $V:\Omega\to\R$ is
continuous. Define
$$
V_x(n)=V(T^nx),\ \ x\in\Omega, n\in\Z.
$$

We now assume that $M$ is a smooth compact manifold and $T$ is a $C^1$ diffeomorphism of $M$. Let $d$ be the Riemann metric on $M.$  Then we have the following simple consequence of super-recurrence. 

\begin{prop} \label{prop.gordon}
 If $x\in M$ is super-recurrent and $V:M\to\R$ is H\"older, then $V_x$ is a Gordon
potential.
\end{prop}

\begin{proof} 
Since $T$
is  $C^1$, $T$ is Lipschitz. Let $L>1$ be the Lipschitz
constant. 
Assume $V$ is $\beta$-H\"older with H\"older constant
$C_1.$ Let $\alpha>1$ be the recurrent exponent of $x$. Let
$\{k_n:n\ge 1\}$ be the sequence related to $x$. By taking a subsequence, we can assume $k_n\ge n$.   For
$1\le l\le k_n$ we have 
\begin{eqnarray*}
|V_x(l)-V_x(l\pm k_n)|&=&|V(T^lx)-V(T^{l\pm k_n}x)|\\
&\le & C_1 d(T^lx,T^{l\pm k_n}x)^\beta\\
&\le & C_1 [L^l d(x,T^{\pm k_n}x)]^\beta\\
&\le & C_1 [L^{2k_n} e^{-k_n^\alpha}]^\beta\\
&=& C_1 \exp\left(-\beta(k_n^\alpha-2k_n\ln L)\right)\\
&\le& n^{-k_n}
\end{eqnarray*}
as soon as $n$ is big enough. By the definition, $V_x$ is a Gordon
potential. \end{proof}

Hence, the second part   of Theorem \ref{application} follows from the first part of Theorem \ref{application} and   Proposition \ref{prop.gordon} and  the above mentioned spectral consequence of the Gordon condition on the potentials.

We now proceed to the construction of the reparametrized flow.

\section{Super-recurrent mixing flows} \label{three}

We start with some notations and reminders on reparametrizations and special flows.

\subsection{Translation flows on the torus} The translation flow on ${\T}^n$ of vector ${\a} \in
{\R}^n$ is the flow arising from the constant vector field
$X({x}) = {\a}$. We denote this flow by $\lbrace
R_{t \a} \rbrace$. When  the numbers
$1,\a_1,\cdots$ $,\a_n$ are rationally independent, i.e. none of them  is a
rational combination of the others, $\lbrace R_{t \a} \rbrace$ is uniquely ergodic for the Haar measure $\mu$ on the torus. In this case we say
it is an irrational flow. 

\subsection{Reparametrized flows} \label{sec.rep} If
$\phi$ is a strictly positive smooth real function on ${\T}^n$, we define  the reparametrization of   $\lbrace R_{t \a} \rbrace$ with velocity $\phi$ as the flow given by the vector field $\phi(x) {\a}$,
that is, by the system 
$${ d {x}\over dt } = \phi ({x}) {\a}.$$
The new flow has the same orbits as   $\lbrace R_{ t \a} \rbrace$ and
preserves a measure equivalent to the Haar measure given by the
density ${1 \over \phi}$. Moreover, if  $\lbrace R_{t \a} \rbrace$ is
uniquely ergodic then so is the reparametrized flow (see \cite{parry}).

\subsection{ Special flows.} The reparametrizations of linear flows can be viewed as special flows above toral translations. We give the formal definition.

\begin{defi}\label{special-flow}
Given a Lebesgue space $L$, a measure preserving transformation $T$ on
$L$ and an integrable strictly positive real function $\varphi$ defined on $L$ we define the special
 flow over $T$ and under the {\sl ceiling function} $\varphi$
by inducing on ${{L \times \R} / \sim }$, where $\sim$ is the
identification $(x, s + \varphi(x)) \sim (T(x),s)$,the action of \begin{eqnarray*}L \times \R &  \rightarrow &  L \times \R \\ (x,s) & \rightarrow & (x,s+t). \end{eqnarray*} \end{defi} 

 If $T$ preserves a unique probability measure $\lambda$ then the special flow will preserve a unique probability measure that is the normalized product measure of $\lambda$ on the base and the Lebesgue measure on the fibers. 

We will be interested in special flows above minimal translations $R_{\a,\a'}$ of the two torus and under smooth functions $\varphi(x,y) \in C^{\infty}(\T^2,{\R_+^*})$ that we will denote by $T^t_{\a,\a',\varphi}$. We denote $M_\varphi =\{ (z,s) : z \in \T^2, s \in [0, \varphi(z) \}$. We will still denote by $\mu$ the product of the Haar measure of $\T^2$ with the normalized Lebesgue measure on the line. 

In all the sequel we will use the following notation, for $m \in \N$,
$$S_m \varphi (z) = \sum_{l=0}^{m-1} \varphi(z+l(\a,\a')).$$

With this notation, given $t \in \R_+$ we have for $\xi \in M_\varphi$, $\xi=(z,s)$
  \begin{equation} \label{equation.special} 
  T^t \xi = \left( R^{N(t,s,z)}_{\a,\a'}(z), t+s -
\varphi_{N(t,s,z)}(z) \right), 
\end{equation}
where  $N(t,s,z)$ is the largest integer $m$ such that $t+s-\varphi_m(x)
  \geq 0$,
 that is the number of fibers covered by  $(z,s)$ during its
  motion under the action of the flow until time $t$.

\subsection{Mixing} We also recall the definition of mixing for a measure
preserving flow: a flow $\lbrace
T_t \rbrace$ preserving a measure $\nu$ on $M$ is said to be mixing if, 
for any measurable subsets $A$ and $B$ of $M$, one has 
\begin{eqnarray*}
\lim_{t\rightarrow \infty} \nu ( T^tA \bigcap B )  = \nu(A) \nu (B). 
\end{eqnarray*}

By standard equivalence between special flows and reparametrizations (see for example \cite{mixing}), Theorem \ref{application} follows from 

\begin{theo} \label{special} There exists a vector $(\a,\a') \in \R^2$ and a smooth strictly positive function $\varphi$ defined over $\T^2$ 
such that  the special flow $T^t_{\a,\a',\varphi}$  is mixing and $\mu$-a.e. $\xi \in M_\varphi$ is super-recurrent for $T^1_{\a,\a',\varphi}$.
\end{theo}

%\begin{rema} the spectral type of the flow is purely singular. 
%The equivalence between Theorems \ref{special} and \ref{application} is proved in \cite{mixing}, Section 4.
%\end{rema} 

We will now undertake the construction of the special flow $T^t_{\a,\a',\varphi}$, following the same steps as \cite{singular}. We will first choose a special translation vector on $\T^2$, then we will give two criteria on the Birkhoff sums of the special function $\varphi$ above $R_{\a,\a'}$ that will guarantee mixing and super-recurrence respectively.  Finally we build a smooth function $\varphi$ satisfying these criteria. 

\subsection{Choice of the translation on $\T^2$.} \label{R}

We start with a quick reminder on continued fractions. Let $\a \in \R \setminus \Q$. There exists a sequence of  rationals 
${\lbrace {p_n \over
  q_n} \rbrace}_{n \in \N} $ , called the convergents of $\a$, such that:
\begin{eqnarray} \label{best}  \parallel q_{n-1} \a \parallel < \parallel k \a \parallel \ \ \forall k < q_{n} \end{eqnarray} 
(where $\|x\|:=\min\{|x-n|: n\in\Z\}$) and for any $n$ 

\begin{eqnarray} \label{equation} 
  { 1  \over q_n (q_n + q_{n+1})} \leq {(-1)}^n ( \a - {p_n \over q_n}) \leq   { 1  \over q_n q_{n+1}}.
\end{eqnarray}

We recall also that any
irrational number $\a \in \R - \Q$ has a writing in continued fraction
$$ \a = [a_0,a_1,a_2,...]= a_0 + 1 / (a_1 + 1 / (a_2+...)),$$               
where ${\lbrace a_i \rbrace}_{i \geq 1}$ is a sequence of integers
$\geq 1$,
$a_0 = [\a]$. Conversely  any sequence  ${\lbrace a_i \rbrace}_{i \in \N}$
 corresponds to a unique number $\a$. The convergents of $\a$ are
given by the $a_i$ in the following way:
\begin{eqnarray*}
p_n = a_n p_{n-1} + p_{n-2} \ \ &{\rm for}& \ \ n \geq 2, p_0=a_0,
p_1=a_0a_1 + 1, \\
q_n = a_n q_{n-1} + q_{n-2} \ \ &{\rm for}& \ \ n \geq 2, q_0=1,
q_1=a_1. 
\end{eqnarray*}

Following \cite{Y} and as in \cite{mixing}, we take $\a$ and $\a'$ satisfying 
\begin{eqnarray}
q'_{n} &\geq& e^{(q_n)^5}, \label{111} \\
q_{n+1} &\geq& e^{(q'_n)^5}. \label{222}
\end{eqnarray}

Vectors $(\a,\a') \in \R^2$ satisfying (\ref{111}) and (\ref{222}) are obtained by an adequate choice of the sequences  $a_n(\a)$ and $a_n(\a')$. Moreover, it is easy to see that the set of vectors satisfying (\ref{111}) and (\ref{222}) is a continuum (Cf. \cite{Y}, Appendix 1).

\subsection{Mixing criterion}
 
We will use the criterion on mixing for a special flow $T^t_{\a,\a',\varphi}$ studied in \cite{mixing}. It is based on the uniform stretch of the Birkhoff sums $S_m \varphi$ of the ceiling function above the $x$ or the $y$ direction alternatively depending on whether $m$ is far from the $q_n$ or from $q'_n$. From \cite{mixing}, Propositions 3.3, 3.4 and 3.5 we have the following sufficient
 mixing criterion. We denote by $\{x\} \in [0,1)$ the fractional part of a real number $x$. 
 
\begin{prop}[Mixing  Criterion ] \label{criterion mixing} Let $(\a,\a')$ be as in Section \ref{R} and $\varphi \in C^2(\T^2, {\R^*_+})$.
If for every $n \in \N$ sufficiently large, we have a set $I_n $ equal to $[0,1]$ minus a finite number of intervals whose lengths converge to zero such that:

\begin{itemize}
\item $\displaystyle{ m \in \left[{e^{2(q_n)^4} \over 2} , 2 e^{2(q'_n)^4} \right] \Longrightarrow  \left| D_x S_m \varphi (x,y) \right| \geq {m \over e^{(q_n)^4}} { q_n \over n}}$, for any $y \in \T$ and $\lbrace q_n x \rbrace \in I_n$;
 
\vspace{0.1cm} 

\item $\displaystyle{ m \in \left[{e^{2(q'_n)^4} \over 2} , 2 e^{2(q_{n+1})^4} \right]  \ \ \Longrightarrow  \left| D_y S_m \varphi (x,y) \right| \geq {m \over e^{(q'_n)^4}} { q'_n \over n}}$, for any $x \in \T$ and $\lbrace q'_n y \rbrace \in I_n$.

\end{itemize}

Then the special flow $T^t_{\a,\a',\varphi}$ is mixing.

\end{prop}

\subsection{Super-recurrence Criterion}

 We give now a condition on the Birkhoff sums of $\varphi$ above $R_{\a,\a'}$ that is sufficient to insure super-recurrence for $T^1_{\a,\a',\varphi}$.
\begin{prop}[Super-recurrence criterion] \label{criterion approximations}
If for $n$ sufficiently large, we have for any $x$ such that $1 / n^2 \leq \lbrace q_n x \rbrace \leq 1 /n - 1 / n^2$  and for any $y \in \T$ 
\begin{eqnarray} \label{r1} \left| S_{q_nq'_n} \varphi(x,y) - q_nq'_n  \right| \leq {1 \over e^{(q_nq'_n)^2}},
\end{eqnarray}
then $\mu$-almost every $z =(x,y,s) \in M_\varphi$ is super recurrent for the time one map of   the special flow $T^t_{\a,\a',\varphi}$ as in Definition \ref{special-flow}.

\end{prop}

\begin{proof} Denote by $T^t$ the flow $T^t_{\a,\a',\varphi}$ and let $t_n=q_nq'_n$. From \eqref{equation.special} we have that 
$$T^{t_n} (x,y,s) = (x+t_n \a,y+t_n \a', s+t_n-S_{t_n} \varphi(x,y)).$$
From \eqref{equation}, \eqref{111} and \eqref{222} we get that $\|t_n \a \|, \|t_n \a' \|\leq  {1 \over e^{t_n^3}}$. Now, for $x$ such that  $1 / n^2 \leq \lbrace q_n x \rbrace \leq 1 /n - 1 / n^2$, for any $y \in \T$ and for $s \in [0,\varphi(x,y))$ 
we have from \eqref{r1} that  $z=(x,y,s)$ satisfies (for the Euclidean distance) $d(z,T^{t_n}z)\leq  {2 \over e^{t_n^2}}$.

Now, the set ${\mathcal C}_n =\{ x \in \T : 1 / n^2 \leq \lbrace q_n x \rbrace \leq 1 /n - 1 / n^2 \}$ has Lebesgue measure larger than $1/2n$. As $q_n$ increases very fast, we have that the sets $\mathcal C_n$ are almost independent, from which it follows by Borel-Cantelli type lemmas that Lebsgue a.e. $x \in \T$ belongs to infinitely many of the  $\mathcal C_n$. Thus almost every $z \in M_\varphi$ is super-recurrent. 
\end{proof} 

%For $k \leq q_n-1$ define $x_k=k/q_n +1/2nq_n
%$. For $y_0 \in \T$  denote by $B_n(k,y_0)$ the ball centered at $(x%_k,y_0,0)$ of radius $1/nq_n$. From (\ref{r1}) and the definition of %special flows we have for $(x,y,s) \in B_n(k,y_0)$
%$$T^{q_nq'_n}(x,y,s) = (x+q_nq'_n \a, y+ q_nq'_n \a', s +  S_{q_nq'_n%} \varphi(x,y) - q_nq'_n)$$
%but $|||q_nq'_n \a||| \leq q'_n / q_{n+1} = o(e^{-q_n})$ as well as 
%  $|||q_nq'_n \a'||| \leq q_n / q'_{n+1} = o(e^{-q_n})$ hence (\ref{r%1})  implies that $ \mu \left( T^{q_nq'_n} B_n(k,y_0) \triangle B_n(k,y_0) \right) =O(1 / q_n^4)= o(e^{-n} \mu(B_{n,i}))$.  The same is true% for the balls $T^s(B_n(k,y_0))$ where $s \leq \sup_{(x,y,0) \in B_n%(k,y_0)} - {1 \over q_n}$. 

%On the other hand it is clear from the difference of scale between the successive terms of the sequence  $q_n$ that the sets $C_n$ are almost independent and the fact that $\mu(C_n) \geq {1 \over n} \inf_{(x,y) \in \T^2} \varphi(x,y)$ then implies by the Borel Cantelli Lemma that $\mu  \left( \mathop{\bigcap} \limits_{m \in \N} \mathop{\bigcup} \limits_{n \geq m} {C}_n \right) = 1$. Condition (iii) of the Definition \ref{definition} is hence established. \carre

\subsection{Choice of the ceiling function $\varphi$.} \label{phi}

Let $(\a,\a')$ be as above and define 
$$ f(x,y) = 1+ \sum_{n \geq 2} X_n(x) + Y_n(y)$$ 
where 
\begin{eqnarray}
\label{XX} X_n(x) &=& {1 \over e^{(q_n)^4}} \cos (2\pi q_nx) \\
\label{YY} Y_n(y) &=& {1 \over e^{(q'_n)^4}} \cos (2 \pi q'_n y).
\end{eqnarray}

Using Criterion \ref{criterion mixing}, we proved in \cite{mixing} that the flow $T^t_{\a,\a',f}$ is mixing (the proof will be recalled below). In order to keep this criterion valid but have in addition the conditions of Criterion \ref{criterion approximations} satisfied we modify the ceiling function in the following way:

$\bullet $ We keep $Y_n(y)$ unchanged.

\vspace{0.1cm} 

$\bullet$ We replace $X_n(x)$ by a trigonometric polynomial $\tilde{X}_n$ with integral zero, that is essentially equal to $0$ for $\lbrace q_nx \rbrace < 1 /n$ and  whose  derivative has its absolute value  bounded from below by ${q_n / e^{(q_n)^4}}$ for 
 $ \lbrace q_n x \rbrace \in [0,1]\setminus \bigcup_{j=0}^4 [j/4-2/n,j/4+2/n]$.
The first two properties of $\tilde{X}_n$ will yield Criterion \ref{criterion approximations} while the  lower bound on the absolute value of its derivative will insure Criterion \ref{criterion mixing}.

More precisely, the following Proposition enumerates some properties that we will require on $\tilde{X}_n$ and its Birkhoff sums, and that will be sufficient for our purposes.

\begin{prop} \label{xtilde} Let $(\a,\a')$ be as in Section \ref{R}. There exists a sequence of trigonometric polynomials $\tilde{X}_n(x)$ satisfying

\begin{enumerate}

\item[(1)] \label{x1}  $\displaystyle{ \int_{\T} \tilde{X}_n(x) dx=0}$;

\item[(2)]  \label{x2} For any $r \in \N$, for every $n \geq N(r)$, $\displaystyle{ {\parallel \tilde{X}_n \parallel}_{C^r} \leq {1 \over e^{(q_n)^4\over 2}}}$;

\item[(3)]  \label{x3} For $\displaystyle{ \lbrace q_n x \rbrace \leq {1 \over n} }$, 
 $\displaystyle{ |\tilde{X}_n(x)| \leq {1 \over e^{(q_n q'_n)^4} } }$;
  
\item[(4)]  \label{x4} For $\displaystyle{\lbrace q_n x \rbrace \in [{2 \over n},  {1 \over 4} - {2 \over n}] \cup [{3 \over 4} +{2 \over n},  {1} - {2 \over n}]     }$, it holds  ${\displaystyle{\tilde{X}_n'(x)  \geq {q_n \over e^{(q_n)^4}}}}$, 
as well as
\newline  for  $\displaystyle{    \lbrace q_n x \rbrace \in [{1 \over 4} +{2 \over n}, {1 \over 2} - {2 \over n}] \cup [{1 \over 2} +{2 \over n}, {3 \over 4} - {2 \over n}]       }$,  it holds {$\displaystyle{\tilde{X}_n'(x)  \leq  -{q_n \over e^{(q_n)^4}}}$;}

\item[(5)]  \label{x5} $\displaystyle{\parallel S_{q_n} \sum_{l\leq n-1} \tilde{X}_{l} \parallel_{C^0} \leq {1 \over e^{(q_n q'_n)^4} } }$; 
 
\item[(6)]   \label{x6} For any $m \in \N$, $\displaystyle{\parallel S_{m} \sum_{l\leq n-1} \tilde{X}'_{l} \parallel_{C^0} \leq q_n}$.

%\item[(7)] \label{x7} For $k \in \Z$, the Fourier coefficients of $\tilde{X}_n$ satisfy $\displaystyle{ |\hat{\tilde{X}}_{n,k}| \leq {1 \over n^2 e^{|k|}}}$.

\end{enumerate}

\end{prop}

Before we prove this Proposition let us show how it allows to produce the example of Theorem \ref{special}. Define for some $n_0 \in \N$

\begin{eqnarray} \label{varphi} 
\varphi(x,y) = 1+ \sum_{n=n_0}^{\infty} \tilde{X}_n(x) + Y_n(y)
\end{eqnarray}
that is of class $C^\infty$ from Property 
%(\ref{x2}) 
(2) of $\tilde{X}_n$ and 
from the definition of $Y_n$ in (\ref{YY}). From (2) again, we can choose $n_0$ sufficiently large so that $\varphi$ is strictly positive.  Furthermore, we have

\begin{theo} Let $(\a,\a') \in \R^2$ be as in Section \ref{R} and $\varphi$ be given by (\ref{varphi}). Then the special flow $T^t_{\a,\a',\varphi}$ satisfies the conditions of Propositions \ref{criterion mixing} and \ref{criterion approximations} and hence the conclusion of Theorem \ref{special}.
 \end{theo}

\noindent {\sl Proof.}  The second part of Proposition \ref{criterion mixing} is valid exactly as in \cite{mixing} since $Y_n$ has not been modified. We sketch its proof for completeness. For $m\in \N$, we have that 
$$S_m Y_n(y)= {\rm Re} \left({Y(m,n) \over e^{(q'_n)^4}} e^{i2 \pi {q'_n}y} \right), $$
with 
$$ Y(m,n)= {{1-e^{i2 \pi m{q'_n} \a'}} \over {1-e^{i2 \pi q'_n\a'}}}.$$

It follows from \eqref{best}--\eqref{222} that 

\begin{align} \label{nnn1}
|Y(m,k)|&\leq m, \ \forall k \in \N^*,  \\
|Y(m,k)|&\leq q'_{n}, \ \forall k < n,  \label{nnn2} \\
  |Y(m,n)|&\geq { 2 \over \pi} m,  \ \forall m\leq {q'_{n+1} \over 2}.  \label{nnn3}
\end{align}
Let now $m \in [e^{2(q'_{n})^4} / 2, 2e^{2(q_{n+1})^4}]$, and $y$ be such that $3 / n \leq \lbrace q'_n y \rbrace \leq 1 / 2  - 2 / n.$ Then  \eqref{equation} and \eqref{111} imply that for any $0\le l\leq m$ : $2 / n \leq  \lbrace q'_n (y+l\a') \rbrace \leq 1 / 2  - 1 / n$. Hence \eqref{nnn3} implies that 
$$|S_m {Y}_n'(y)| \geq {8m q'_n \over ne^{(q'_n)^4}}.$$

Using \eqref{nnn1} and \eqref{nnn2} to bound $\| \sum_{k>n} S_m {Y}_k'\|$ and $\|\sum_{k<n} S_m {Y}_k'\|$ respectively, we obtain that $|S_m \sum_{k\geq 1} {Y}_k'(y)| \geq {m q'_n \over {n}e^{(q'_n)^4}}$. The case 
$1/2+3 / n \leq \lbrace q'_n y \rbrace \leq 1   - 2 / n$ is treated similarly.

We now turn to the control of the Birkhoff sums in the $x$ direction. 
Let $m \in [e^{2(q_n)^4} / 2, 2e^{2(q'_n)^4}]$ and $x$ be such that $| \lbrace q_n x \rbrace - j/4| > 3/n$ for any positive integer $j \leq 4$. For definiteness assume that $\lbrace q_n x \rbrace \in [{3 \over n},  {1 \over 4} - {3 \over n}],$ the other cases being similar.  

From \eqref{222} we get for any $0\le l \leq m$ that $2 / n \leq  \lbrace q_n (x+l\a) \rbrace \leq 1 / 4  - 2 / n,$ hence by Property (4) of $\tilde{X}_n$ 
$$|S_m \tilde{X}_n'(x)| \geq {m q_n \over  e^{(q_n)^4}}.$$
 On the other hand, Properties (2) and (6) imply that 
\begin{eqnarray*} \parallel S_m \varphi'- S_m \tilde{X}_n'  \parallel &\leq& q_n + m \sum_{l \geq n+1} {1 \over e^{{(q_l)^4 \over 2}}} \\
&\leq& q_n + {2m \over e^{{(q_{n+1})^4 \over 2}}} \\
&=& o({mq_n \over e^{(q_n)^4}}) \end{eqnarray*}
for the current range of $m$. The criterion of Proposition \ref{criterion mixing} thus holds  true.

Let now $x $ be as in Proposition \ref{criterion approximations}, that is $1 / n^2 \leq \lbrace q_n x \rbrace \leq 1 /n - 1 / n^2$. From \eqref{222} we have for any $l \leq q_n q'_n$ that $0 \leq \lbrace q_n(x+l\a) \rbrace \leq 1 / n$, hence  Property (3) implies 
\begin{eqnarray} |S_{q_nq'_n} \tilde{X}_n (x) | \leq  {q_n q'_n \over e^{(q_n q'_n)^4}} \leq {1 \over e^{(q_n q'_n)^3}}. \label{101} \end{eqnarray}
 From Properties (2) and (5) we get for $n$ sufficently large
\begin{eqnarray} \parallel S_{q_nq'_n} \sum_{l \neq n}   \tilde{X}_l   \parallel &\leq&  { q'_n \over e^{(q_n q'_n)^4} }+ q_nq'_n \sum_{l \geq n+1} {1 \over e^{{(q_l)^4 \over 2}}} \nonumber  \\
&\leq&   {  {1 \over e^{(q_n q'_n)^3} }}.  \label{102}
\end{eqnarray}

On the other hand, it follows from the definition of convergents in Section \ref{R} and (\ref{equation}) that for any $y \in \T$, for any $|j| < q'_{n}$, we have 
\begin{eqnarray}
|S_{q'_n} e^{i2 \pi jy}| &=& \left| \sin(\pi j q_n' \a') \over \sin(\pi j \a') \right|  \nonumber \\
&\leq&  {2\pi j q_n' \over q'_{n+1}}, \label{103} 
\end{eqnarray}
which, using \eqref{111} and \eqref{222},  yields for $Y_l$ as in (\ref{YY})
\begin{eqnarray} \| S_{q'_n} \sum_{l<n} Y_l \| \leq  {1 \over e^{(q_n q'_n)^3} }, \label{104} \end{eqnarray} 
while clearly 
\begin{eqnarray} \| S_{q'_n} \sum_{l > n} Y_l \| \leq {  e^{-{(q'_{n+1})^4 \over 2} } }  \leq  {1 \over e^{(q_n q'_n)^3} } \label{105} \end{eqnarray}
and 
\begin{eqnarray} \| S_{q'_n} Y_n \| \leq {  q'_n \over e^{(q'_n)^4} }   \leq  {1 \over e^{(q_n q'_n)^3} }. \label{1005} \end{eqnarray}
Putting together \eqref{104}--\eqref{1005} yields 
\begin{eqnarray} \| S_{q_n q'_n} \sum_{l=n_0}^\infty  Y_l \| \leq     {1 \over 2 e^{(q_n q'_n)^2} } \label{10005} \end{eqnarray}

In conclusion, (\ref{r1}) follows from (\ref{101}), (\ref{102}), and (\ref{10005}). \carre

It remains to construct $\tilde{X}_n$ satisfying (1)-(6).

\subsection{Proof of Proposition \ref{xtilde}}

Consider on $\R$ a $C^\infty$ function, $0 \leq \theta \leq 1$ such that
\begin{eqnarray*}
\theta(x) &=& 0 \ {\rm for } \ x \in (-\infty,0] \\
\theta(x) &=&  1  \ {\rm for } \ x \in [1, + \infty ).
\end{eqnarray*}

Then we define 
$$\theta_n(x) := \theta \left(nq_n(x-{1 \over nq_n} )\right) - \theta \left( nq_n(x-{1 \over 4 q_n}+{2\over nq_n})\right).$$
Observe that 
$$ \theta_n(x) = \left\{\begin{array}{l}  1,  {\rm \ for \ }  x \in [{2 \over n q_n}, {1 \over 4q_n} - {2 \over n q_n}] \\
0,  {\rm \ for \ } x \in [-\infty, {1 \over nq_n} ] \cup [{1 \over 4q_n}-{1 \over n q_n},+\infty] 
     \end{array} \right. $$

We define on $\R$ the following functions 
$$U_n(x)= \int_{-\infty}^x \theta_n(u)du$$
then
$$V_n(x)=U_n(x)-U_n(x-{1 \over 4q_n}).$$

Observe that $V_n$ is compactly supported inside $[0,{1 \over 2q_n}]$ and has derivative equal to $1$ for $x \in J_n =[{2 \over n q_n}, {1 \over 4q_n} - {2 \over n q_n}]$ and derivative equal to $-1$ on ${1 \over 4q_n}+J_n$. 
Define now the zero averaged function supported inside $[0,{1\over q_n}]$ 
$$W_n(x)=V_n(x)-V_n(x-{1 \over 2q_n}).$$
The derivative of $W_n$ is constant equal to $1$ on $J_n \cup ({3 \over 4q_n}+J_n)$ and constant equal to $-1$ on  $({1 \over 4q_n} +J_n) \cup ({1 \over 2q_n} +J_n)$. Also $W_n \equiv 0$ on  $[0,{1 \over nq_n}]\cup [{1 \over q_n}-{1 \over nq_n},{1 \over q_n}]$.

We define the following function on the circle  $\T=\R / \Z$
$$\hat{X}_n(x) := \frac{q_n}{e^{(q_n)^4}} \sum_{k=0}^{q_n-1}  W_n(x+\frac{k}{q_n}).$$
 It is easy to check (1),(2),(3) and (4) of Proposition \ref{xtilde} for $\hat{X}_n$.

Now we consider the Fourier series of $\hat{X}_n(x) = \sum_{k \in \Z}  \hat{X}_{n,k} e^{i2\pi k x}$
and let

$$\tilde{X}_n (x):= \sum^{q_{n+1}-1}_{k=-q_{n+1}+1} \hat{X}_{n,k} e^{i 2 \pi k x}.$$

From the order of the truncation and the $C^r$ norms of $\hat{X}_n$ it is easy to deduce that for any $r \in \N$
$${\| \tilde{X}_n - \hat{X}_n \|}_{C^r} \leq   {1 \over e^{(q_n q'_n)^5} }, $$
which allows to check (1), (2), (3) and (4) for $\tilde{X}_n$.

\vspace{0.2cm} 

\noindent  {\sl Proof of Property (5).} As for (\ref{103}), using the definition of convergents in Section \ref{R}, we obtain  for any $x \in \T$, and for any $|k| < q_{n}$  
\begin{eqnarray*}
|S_{q_n} e^{i2 \pi kx}| &\leq&  {2\pi k q_n \over q_{n+1}},
\end{eqnarray*}
hence for $\tilde{X}_l := \sum^{q_{l+1}-1}_{k=-q_{l+1}+1} \hat{X}_{l,k} e^{i 2 \pi k x}$ and $l \leq n-1$ we have
\begin{eqnarray*}
\| S_{q_n} \tilde{X}_l \| &\leq&   {2\pi q_n^2 \over q_{n+1}} \sum^{q_{l+1}-1}_{k=-q_{l+1}+1} |\hat{X}_{l,k}| \\ 
&\leq&
{4 \pi  q_n^3 \over q_{n+1}} \| \hat{X}_l \|,
\end{eqnarray*}
thus, Property (5) follows.

\vspace{0.2cm} 

\noindent {\sl Proof of Property (6).} For any $|k| < q_n$ we have  
\begin{eqnarray*}
|S_{m} e^{i2 \pi kx}| &=& \left| \sin(\pi m k  \a) \over \sin(\pi k \a) \right| \\
&\leq&  {1 \over |\sin(\pi k \a)} \\
&\leq&  q_n.
\end{eqnarray*}
Thus, for $l \leq n-1$, we use that $  |\hat{X}_{l,k} | \leq {1 \over {(2 \pi |k|)}^3} {\| D^3_x \hat{X}_l \|}$ and get that 
\begin{eqnarray*}
\| S_{m} \tilde{X}_l' \| &\leq&   
 \sum^{q_{l+1}-1}_{k=-q_{l+1}+1} {1 \over {(2 \pi |k|)}^2}  q_{l+1} {\| D^3_x \hat{X}_l \|} \\ &\leq&
{1 \over 12} q_{l+1}  \| D^3_x \hat{X}_l \| 
\end{eqnarray*}
from which Property (6) follows.   \hfill $\Box$ 

$$ $$

\end{document}